\documentclass{amsart}
\usepackage{txfonts}
\usepackage{lipsum}
\usepackage[colorlinks, linkcolor=black, citecolor=blue, linktocpage,urlcolor=blue]{hyperref}
\numberwithin{equation}{section}
\theoremstyle{definition}
\newtheorem{thm}{Theorem}[section]
\newtheorem{definition}[thm]{Definition}
\newtheorem{proposition}[thm]{Proposition}

\newtheorem{lemma}[thm]{Lemma}
\newtheorem{corollary}[thm]{Corollary}

\newtheorem{remark}[thm]{Remark}

\begin{document}
	\setlength{\abovedisplayskip}{10pt}
	\setlength{\belowdisplayskip}{5pt}
	
	\newcommand{\Ext}{\bigwedge\nolimits}
	\newcommand{\Div}{\operatorname{div}}
	\newcommand{\Hol} {\operatorname{Hol}}
	\newcommand{\diam} {\operatorname{diam}}
	\newcommand{\Scal} {\operatorname{Scal}}
	\newcommand{\scal} {\operatorname{scal}}
	\newcommand{\Ric} {\operatorname{Ric}}
	\newcommand{\Hess} {\operatorname{Hess}}
	\newcommand{\grad} {\operatorname{grad}}
	\newcommand{\Rm} {\operatorname{Rm}}
	\newcommand{ \Rmzero } {\mathring{\Rm}}
	\newcommand{\Rc} {\operatorname{Rc}}
	\newcommand{\Curv} {S_{B}^{2}\left( \mathfrak{so}(n) \right) }
	\newcommand{ \tr } {\operatorname{tr}}
	\newcommand{ \Riczero } {\mathring{\Ric}}
	\newcommand{ \Ad } {\operatorname{Ad}}
	\newcommand{ \dist } {\operatorname{dist}}
	\newcommand{ \rank } {\operatorname{rank}}
	\newcommand{\Vol}{\operatorname{Vol}}
	\newcommand{\dVol}{\operatorname{dVol}}
	\newcommand{ \zitieren }[1]{ \hspace{-3mm} \cite{#1}}
	\newcommand{ \pr }{\operatorname{pr}}
	\newcommand{\diag}{\operatorname{diag}}
	\newcommand{\Lagr}{\mathcal{L}}
	\newcommand{\av}{\operatorname{av}}
	\newcommand{ \floor }[1]{ \lfloor #1 \rfloor }
	\newcommand{ \ceil }[1]{ \lceil #1 \rceil }
	\newcommand{\Sym} {\operatorname{Sym}}
	\newcommand{\bcirc}{ \ \bar{\circ} \ }
	\newcommand{\sign}[1]{\operatorname{sign}(#1)}
	\newcommand{\cone}{\operatorname{cone}}
	\newcommand{\pbd}{\varphi_{bar}^{\delta}}
	\newcommand{\End}{\operatorname{End}}
	
	\renewcommand{\labelenumi}{(\alph{enumi})}
	\newtheorem{maintheorem}{Theorem}[]
	\renewcommand*{\themaintheorem}{\Alph{maintheorem}}
	\newtheorem*{remark*}{Remark}
	\title{Manifolds with harmonic Weyl curvature and curvature operator of the second kind}
	
	%    Remove any unused author tags.
	
	%    author one information
	\author{Haiping Fu}
	\address{Department of Mathematics, Nanchang University, Nanchang 330031, People’s Republic of China}
	\curraddr{}
	\email{mathfu@126.com}
	\thanks{Supported in part by National Natural Science Foundations of China \#12461008 and 12271069,  Jiangxi Province
		Natural Science Foundation of China \#20202ACB201001, Jiangxi Province Graduate Student Innovation Special Fund Project \#YC2025-B037.}
	
	%    author two information
	\author{Yao Lu}
	\address{Department of Mathematics, Nanchang University, Nanchang 330031, People’s Republic of China}
	\curraddr{}
	\email{luyao@email.ncu.edu.cn}
	\thanks{}
	
	
	
	\date{}
	
	\dedicatory{}
	
	\begin{abstract}
		We prove that a compact Riemannian manifold of dimension $n\ge 8$ with harmonic Weyl curvature and $\frac{3(n-1)(n+2)}{4(3n-1)}$-nonnegative curvature operator of the second kind is either globally conformally equivalent to a space of positive constant curvature or is isometric to a flat manifold. In particular, We also give a classification of
		four-dimensional manifolds with harmonic Weyl curvature satisfying a cone condition. This result generalizes the work in \cite{DFY24,FLD,Li22}.
	\end{abstract}
	\maketitle
	
	\section{introduction}
	
	A central theme in Riemannian geometry is the study of the relationship between curvature operator and the topological structure of Riemannian manifolds.
	A well-known theorem of Tachibana \cite{Tac74} showed that any compact Riemannian manifold of dimension $n\ge3$ with harmonic curvature and positive(resp., nonnegative) curvature operator is isometric to a quotient of  the standard sphere(resp., locally symmetric). This result was improved by Tran \cite{Tra17}, who proved that a compact Riemannian manifold of dimension $n\ge 4$ with harmonic Weyl curvature and  positive curvature operator is locally conformally flat. In \cite{PW22}, Petersen and Wink showed that a compact Riemannian manifold with harmonic Weyl tensor and $\left[\frac{n-1}{2}\right]$-nonnegative curvature operator is either globally conformally equivalent to a quotient of the standard sphere or locally symmetric; the former case is guaranteed when the curvature operator is $\left[\frac{n-1}{2}\right]$-positive. Subsequently, in \cite{PW21}, they established a related Tachibana-type theorem under the stronger Einstein condition. More recently, Colombo, Mariani, and Rigoli \cite{CMR24} proved that a compact Riemannian manifold of dimension $n\ge3$ with harmonic curvature and $\left[\frac{n-1}{2}\right]$-positive curvature operator must have constant sectional curvature, generalizing Tachibana’s classical result. This finding was further refined by Bettiol and Goodman \cite{BG24}, who showed the condition can be relaxed to $\frac{n-1}{2}$-positive curvature operator for $n\ge5$, and to $\frac{n}{2}$-positive curvature operator for $n=3,4$.
	
	There is another way to define the curvature operator induced by Riemannian curvature tensor, known as the curvature operator of the second kind, see section 2.1 blow for details. In 1986, Nishikawa conjecture \cite{Nis86} states that a closed Riemannian manifold with positive (resp., nonnegative) curvature operator of the second kind is diffeomorphic to a spherical space form (resp., a Riemannian locally symmetric space). This conjecture was later proved by Cao-Gursky-Tran \cite{CGT23} and further refined by Li \cite{Li22,Li24}. Both proofs relied on Brendle’s convergence result \cite{Brendle} for closed Ricci flows with PIC1 toward constant sectional curvature. 
	
	For certain special classes of manifolds, Kashiwada \cite{Kas93} showed that Riemannian manifolds with harmonic curvature and nonnegative (resp., positive) curvature operator of the second kind  are locally symmetric spaces (resp., constant curvature spaces) by using Bochner tenchique.
	Later, Nienhaus, Petersen, and Wink \cite{NPW23} derived a new Bochner formula for the curvature operator of the second kind and used it to prove that $n$-dimensional compact Einstein manifolds with $k(<\frac{3n(n+2)}{2(n+4)})$-nonnegative curvature operators of the second kind must be either flat or a rational homology sphere. Subsequently, Dai and Fu \cite{DF24} also obtained a new Bochner formula and employed it to prove that any compact Einstein manifold must have constant curvature if its curvature operator of the second kind is $k$-nonnegative, where $k=1$ for $4\leq n\leq7$, $k=2$ for $8\leq n\leq10$ and  $k=[\frac{n+2}{4}]$ for $n\geq11$. In 2024, Dai-Fu-Yang	\cite{DFY24} improved upon Kashiwada’s result by proving that a compact manifold with harmonic Weyl tensor and nonnegative curvature operator of the second kind is either globally conformally equivalent to a space of positive constant curvature or is isometric to a flat manifold. Recently, Fu, Lu and Dai \cite{FLD} show a rigidity theorem for manifolds with harmonic curvature, which generalizes the results of \cite{Kas93,DF24}. For further developments on this topic, we refer the reader   to \cite{Li23a, Li24b, Li25, HZ25, FL}.
	
	Let the nondecreasing sequence $\lambda_{i}$ be the eigenvalues of the curvature operator. If $${\lambda }_1+{\lambda }_2+\cdots {\lambda }_{[k]}+(k-[k])\lambda _{[k]+1}> 0(\ge 0),$$ 
	we say that the curvature operator is $k$-positive(resp., nonnegative).
	In this paper, we study Riemannian manifolds of dimension  $n\ge4$ that possess harmonic Weyl curvature and satisfy conditions on the curvature operator of the second kind. Our main results are the following three theorems.
	\begin{maintheorem}\label{A}
		Let $(M, g)$ be an $n(\ge 8)$-dimensional compact Riemannian manifold with harmonic Weyl curvature tensor. If the curvature operator of the second kind $\mathring{R}$ is $\frac{3(n-1)(n+2)}{4(3n-1)}$-nonnegative, then $M$ is either globally conformally equivalent to a space of positive constant curvature or is isometric to a flat manifold.
	\end{maintheorem}
	\begin{corollary}
		Let $(M, g)$ be an $n(\ge 8)$(resp., $11$)-dimensional compact Riemannian manifold with harmonic Weyl curvature tensor. If the curvature operator of the second kind $\mathring{R}$ is $2$(resp., $3$)-nonnegative, then $M$ is either globally conformally equivalent to a space of positive constant curvature or is isometric to a flat manifold.
	\end{corollary}	
	
	\begin{maintheorem}\label{B}
		Let $(M, g)$ be an $n(= 5,7)$-dimensional compact Riemannian manifold with harmonic Weyl curvature tensor. $\mathring{R}$
		be the curvature operator of the second kind. If\\
		(i)$n=5$ and $\mathring{R}$ is $\frac{252}{169}$-nonnegative;\\
		(ii)$n=7$ and $\mathring{R}$ is $\frac{54}{35}$-nonnegative,\\
		then $M$ is either globally conformally equivalent to a space of positive constant curvature or is isometric to a flat manifold.
	\end{maintheorem}
	\begin{remark}
		Theorem~\ref{A} and \ref{B} generalize the result of Dai, Fu and Yang \cite{DFY24}, who obtained the same conclusion under non-negative conditions. However, Our method does not apply to the case $n=6$. Whether Dai-Fu-Yang’s result can be improved to achieve optimality remains an interesting and open problem. For $n(\ge 8)$(resp., $11$), Theorem~\ref{A} can be regarded
as a refined version of their theorem due to \cite{CGT23,Li24}.
	\end{remark}
	
	Furthermore, we obtain a  classification of four-dimensional Riemannian manifolds with harmonic Weyl curvature.
	
	\begin{maintheorem}\label{C}
		Let $(M,g)$ be a $4$-dimensional compact manifold with harmonic Weyl curvature  whose the curvature operator of the second kind $\mathring{R}\,$ satisfies 
		\begin{equation}\label{cone}
			\lambda_{1}+\lambda_{2}+\lambda_{3}\ge -3\bar{\lambda},
		\end{equation}
where $\bar{\lambda}$ represents the average value of all the eigenvalues of $\mathring{R}\,$. 
		Then  one of the following must be true:\\
		(1) $M$ is conformal flat.\\
		(2) up to rescaling,  $M$ is $\mathbb{CP}^2$ with the Fubini–Study metric.\\
		(3) up to rescaling, $M$ is isometric to a quotient of $\Sigma_1^2\times \Sigma_2^2$ with the product metric, where $k_i(i=1,2)$ is constant Gaussian curvature of $\Sigma_i^2$  and $\frac{s}{2}=k_1+k_2>0$.
	\end{maintheorem}
	\begin{remark}
		For $4$-dimensional manifolds, this theorem improves upon \cite[Theorem D]{FL} and their conclusions due to \cite{FLD}. This theorem can be regarded as an optimal rigidity result under certain conditions.
	\end{remark}
	The paper is organized as follows. In Sect. 2, we give an introduction to the curvature
	operator of the second kind and some Lemmas. In Sect. 3, we present a refined Bochner formula for manifold with harmonic Weyl curvature and prove Theorem A. In Sect. 4, we prove Theorem B. The proofs of Theorems C are presented in Sect. 5.
	
	\section{Preliminaries}
	
	\subsection{Notion and Lemmas}
	In a Riemannian manifold $(M^n,g)$, the Riemannian curvature tensor $R$  induces two symmetric linear operators $\hat{R}$ and $\mathring{R}$. Denote by $V$ be the tangent space of this Riemannian manifold, by ${\Lambda }^2V$  the space of skew symmetric $2$-tensor over $V$, and by $S_0^2(V)$  the space of trace-free symmetric $2$-tensor over $V$. The curvature operator of the first kind is defined by
	\begin{align*}
		& \hat{R}:{\Lambda }^2V\to {\Lambda }^2V \\ 
		& \hat{R}(e_i\wedge e_j)=\frac{1}{2}\sum\limits_{k,l}{R_{ijkl}e_k\wedge e_l},
	\end{align*}
	and the curvature operator of the second kind is defined in \cite{CGT23, NPW23} by
	$$\mathring{R}={\Pr}_{s_0^2(V)}\circ \left.\bar{R} \right|_{S_0^2(V)},$$
	where the operator $\bar{R}$ act on $S^2(V)$ induced from Riemannian curvature tensor $R$ is defined by
	$$\bar{R}(e_i\odot e_j)=\sum\limits_{k,l}{R_{kijl}e_k\odot e_l},$$
	where $e_i\odot e_j=e_i\otimes e_j+e_j\otimes e_i$. 
	The inner product on $S^2(V)$ is defined by
	\[\left\langle A, B \right\rangle= tr(A^T B),\]
	Then $\frac{1}{\sqrt{2}}\{e_i\odot e_j\}_{1\le i<j\le n},\frac{1}{2}\{e_i\odot e_i\}_{i=1,\dots,n}$ is an orthonormal basis for $S^2(V)$. The operator $\mathring{R}$ can be regarded as the curvature operator $\bar{R}$ and its image restricts on $S_0^2(V)$. 
	
	By identifying symmetric $(0,2)$-tensors with self-adjoint endomorphisms of $V$, Nienhaus, Petersen and Wink gives the following definition in \cite{NPW23}.
	\begin{definition}
		Let $T^{(0,k)}(V)$ denote the space of $(0,k)$-tensor space on $V$. For $S\in S^2(V)$ and $T\in T^{(0,k)}(V),$ we define
		\begin{align*}
			& S:T^{(0,k)}(V)\to T^{(0,k)}(V) \\ 
			& (ST)(X_1,\cdots ,X_k)=\sum\limits_{i=1}^{k}{T(X_1,\cdots ,SX_i,\cdots ,X_k)} 
		\end{align*}
		and define $T^{S^2}\in T^{(0,k)}(V)\otimes S^2(V)$ by
		$$\left\langle T^{S^2}(X_1,\cdots ,X_k),S \right\rangle =(ST)(X_1,\cdots ,X_k).$$
	\end{definition} 
	Hence, if $\left\{ S^\alpha \right\}$ is an orthonormal basis for $S_0^2(V)$, then
	\[T^{S_0^2}=\sum\limits_{\alpha=1}^{N}{S^\alpha T\otimes S^\alpha}.\]
	
	Let $\{{\lambda }_\alpha\}$ denote the eigenvalues of the operator $\mathring{R}\,$  and $\left\{ S^\alpha \right\}$ be the corresponding orthonormal eigenbasis of $S_{0}^2(V)$. In what follows, the eigenvalues $\{{\lambda}_\alpha\}$ are assumed to be arranged in non‑decreasing order. Define the action of $\mathring{R}\,$ on the space $T^{S_{0}^{2}}$ as
	\[\mathring{R}\,(T^{S_{0}^{2}})=\sum\limits_{\alpha}{S^\alpha T\otimes \mathring{R}\,(S^{\alpha}),}\]
	then
	\[\left\langle \mathring{R}\,(T^{S_{0}^{2}}),T^{S_{0}^{2}} \right\rangle =\sum\limits_{\alpha}{{\lambda }_\alpha\left| S^{\alpha}T \right|^2}.\]
	\subsection{Weighted principle}		
	To analyze the individual terms of Bochner formula, we apply the powerful weighted-sum calculus developed in \cite{NPW23}, which allows us to estimate finite weighted sums with nonnegative weights.
	\begin{definition}(\cite[Definition 3.1]{NPW23})
		Let $\{\omega_{i}\}_{\alpha=1}^{N}$ be the nonnegative weights of any finite weighted sums. Define
		\[
		\Omega=\max_{1\leq \alpha\leq N}\omega_{\alpha} \quad \mathrm{and} \quad  S=\sum_{\alpha=1}^{N}\omega_{\alpha}\,.
		\]
		We call $S$ the total weight and $\Omega$ the highest weight.
	\end{definition}
	In particular, if $F(R)$ is a (geometric) quantity depending on curvature tensor $R$ and $\mathring{R}$ is curvature operator of the second kind, then we will write
	$$F(R) \ge [\mathring{R},\Omega, S]$$
	provided $F(R)$ is bounded from below by a weighted sum in terms of the eigenvalues of $\mathring{R}$ with highest weight $\Omega$ and total weight $S$.
	
	Let $\{\lambda_{\alpha}\}_{\alpha=1}^{N}$ be the eigenvalues of $\mathring{R}\,$. Following the notation in \cite{NPW23}, we write $[\mathring{R},\Omega,S]$ to denote any finite weighted sums $\sum_{\alpha=1}^{N}\omega_{\alpha}\lambda_{\alpha}$ in terms of $\{\lambda_{\alpha}\}_{\alpha=1}^{N}$ whose weights satisfy highest weight $\Omega$ and total weight $S$. %We have the following
	\begin{lemma}\label{Lem 3.4}(\cite[Lemma 3.4]{NPW23})
		Let  $\{\lambda_{\alpha}\}_{\alpha=1}^{N}$ denote the ordered eigenvalues of $\mathring{R}$. Then, for any integer $1\leq m\leq N$, we have
		$$[\mathring{R},\Omega,S]\geq(S-m\Omega)\lambda_{m+1}+\Omega\sum_{\alpha=1}^{m}\lambda_{i}.$$
		Therefore, $[\mathring{R},\Omega,S]\ge 0$ if $\mathring{R}$ is $\frac{S}{\Omega}$-nonnegative.
	\end{lemma}
	The following lemma will be useful in our proof.
	\begin{lemma}\label{ComputationsWithWeights}(\cite[Lemma 3.3]{NPW23})
		Let $[\mathring{R}, \Omega, S], [\mathring{R}, \widetilde{\Omega}, \widetilde{S}]$ denote weighted sums of eigenvalues of $\mathring{R}$ with highest weights $\Omega, \widetilde{\Omega}$ and total weights $S, \widetilde{S}$, respectively.
		\begin{enumerate}
			\item If $c>0$, then 
			\begin{align*}
				[\mathring{R}, c \Omega, c S] = c \cdot [\mathring{R}, \Omega, S].
			\end{align*}
			\item If $\Omega \leq \widetilde{\Omega},$ then 
			\begin{align*}
				[ \mathring{R} , \Omega, S] \geq [ \mathring{R} , \widetilde{\Omega}, S].
			\end{align*}
			\item \begin{align*}
				[\mathring{R}, \Omega, S] + [\mathring{R}, \widetilde{\Omega}, \widetilde{S}] \geq [\mathring{R}, \Omega + \widetilde{\Omega}, S + \widetilde{S}].
			\end{align*}
		\end{enumerate}
	\end{lemma}
	
	\section{Bochner formula for manifold with harmonic Weyl curvature}
	Let $M$ be an $n$-dimensional complete manifold with harmonic Weyl curvature, endowed with a standard Riemannian metric $\left\langle \text{,} \right\rangle$. Denote by $W$ its Weyl tensor and by $s$ its scalar curvature. For simplicity, we set
	\[R_{sjti}W_{sjkl}W_{tikl}=\alpha (W)\quad\text{and}\quad R_{sikt}W_{sjkl}W_{ijtl}=\beta(W).\]
	\begin{lemma}
		Let $M$ be an $n(\ge 4)$-dimensional complete Riemannian manifold with harmonic Weyl curvature. Then
		\begin{align}\label{laplacian}
			\left\langle \Delta W,W \right\rangle =2R_{lt}W_{ijkl}W_{ijkt}-\alpha (W)-4\beta (W).
		\end{align}
	\end{lemma}
	\begin{remark}Although a similar formula has been proven in \cite{FH}, its form of expression differs. For the sake of completeness, we still present its proof here.
	\end{remark}
	\begin{proof}
		Since $W$ is harmonic, by the second Bianchi identity we obtain (see \cite{FH})
		\[W_{ijkl,h}+W_{ijlh,k}+W_{ijhk,l}=0.\]
		Combining with the above, the Ricci identity gives
		\begin{align*}
			&\left\langle \Delta W,W \right\rangle=W_{ijkl,ss}W_{ijkl}=W_{ijkl}(W_{ijsl,ks}+W_{ijks,ls})\\
			&=2W_{ijkl}(W_{ijks,sl}+R_{lsti}W_{tjks}+R_{lstj}W_{itks}+R_{lstk}W_{ijts}+R_{lsts}W_{ijkt})\\
			&=4R_{lsti}W_{ijkl}W_{tjks}+2R_{lstk}W_{ijkl}W_{ijts}+2R_{lt}W_{ijkl}W_{ijkt}\\
			&=2R_{lt}W_{ijkl}W_{ijkt}-\alpha (W)-4\beta (W).
		\end{align*}
This completes the proof of Lemma 3.1.
	\end{proof}
	Let 
	$$S^{ij}=\frac{1}{2}\sum\limits_{k,l}{(W_{iklj}+W_{ilkj})e_k\odot e_l}.$$ 
	Then $S^{ij}$ is a trace-less symmetric tensor since Weyl tensor is trace-free. Therefore, we can compute:
	\begin{align}\label{estimate1}
		\sum\limits_{i,j}&{<\mathring{R}(S^{ij}),S^{ij}>}=\sum\limits_{i,j}{<\bar{R}(S^{ij}),S^{ij}>}\nonumber\\
		=&\frac{1}{2}\sum\limits_{i,j,k,l,s,t}{(W_{iklj}+W_{ilkj})(W_{istj}+W_{itsj})(R_{kstl}+R_{ktsl})} \nonumber\\ 
		=&\frac{1}{2}(W_{iklj}W_{istj}R_{kstl}+W_{iklj}W_{istj}R_{ktsl}+W_{iklj}W_{itsj}R_{kstl}+W_{iklj}W_{itsj}R_{ktsl}\\
		&+W_{ilkj}W_{istj}R_{kstl}+W_{ilkj}W_{istj}R_{ktsl}+W_{ilkj}W_{itsj}R_{kstl}+W_{ilkj}W_{itsj}R_{ktsl}) \nonumber\\ 
		=&2W_{iklj}W_{istj}R_{kstl}+2W_{iklj}W_{istj}R_{ktsl} \nonumber\\ 
		=&\frac{1}{2}\alpha (W)-4\beta (W)\nonumber 
	\end{align}
	where the last equality follows from the first Bianchi identity, which gives
	\begin{align*}
		W_{iklj}W_{istj}R_{ktsl}=W_{iklj}W_{istj}(R_{klst}+R_{kstl})
		=W_{iklj}W_{istj}R_{klst}-\beta(W)
	\end{align*}
	and 
	\begin{align*}
		&\alpha(W)=W_{ijkl}W_{ijst}R_{klst}=(W_{ilkj}+W_{ikjl})(W_{itsj}+W_{isjt})R_{klst}\\
		&=W_{ilkj}W_{itsj}R_{klst}+W_{ilkj}W_{isjt}R_{klst}+W_{ikjl}W_{itsj}R_{klst}+W_{ikjl}W_{isjt}R_{klst}\\
		&=4W_{iklj}W_{istj}R_{klst}.
	\end{align*}
	
	In order to derive a Bochner formula for Riemannian manifolds with harmonic Weyl curvature, we require the following proposition, which is taken from \cite{DF24}.
	\begin{proposition}[\cite{DF24}]\label{pro1}
		Let $T$ and $\mathring{R}$ denote a smooth algebraic curvature tensor and a curvature operator of the second kind on a Riemannian manifold $(M, g)$ of dimension $n\ge 3$, respectively. Then
		\[<\mathring{R}(T^{S_{0}^{2}}),T^{S_{0}^{2}}>=\frac{2n+32}{n}R_{st}T_{sjkl}T_{tjkl}-5\alpha (T)+4\beta (T)-\frac{16}{n^2}s\left| T \right|^2.\]
	\end{proposition}
	
	With the aforementioned preparations, we arrive at the following Bochner formula.
	\begin{proposition}\label{pro2}
		Let $M$ be an $n(\ge 4)$-dimensional complete Riemannian manifold with harmonic Weyl curvature. Then
		\[3\left\langle \Delta W,W \right\rangle=\left\langle \mathring{R}\,(W^{S_{0}^2}),W^{S_{0}^{2}} \right\rangle+\frac{4(n-8)}{n}R_{lt}W_{ijkl}W_{ijkt}+4\sum\limits_{i,j}{<\mathring{R}(S^{ij}),S^{ij}>}+\frac{16}{n^2}s|W|^2.\]
	\end{proposition}
	\begin{proof}
		From Proposition~\ref{pro1},  we have
		\begin{eqnarray*}
			& <\mathring{R}(W^{S_{0}^2}),W^{S_{0}^2}>=\frac{2n+32}{n}{R_{st}}{W_{sjkl}}{W_{tjkl}}-5\alpha (W)+4\beta (W)-\frac{16}{n^2}s{\left| W \right|}^2.
		\end{eqnarray*}
		Combining this with (\ref{laplacian}) and (\ref{estimate1}), we obtain
		\begin{align*}
			3\left\langle \Delta W,W \right\rangle &=\left\langle \mathring{R}\,(W^{S_{0}^2}),W^{S_{0}^{2}} \right\rangle +\frac{4(n-8)}{n}R_{lt}W_{ijkl}W_{ijkt}+2\alpha (W)-16\beta(W)+\frac{16}{n^2}s|W|^2 \\ 
			& =\left\langle \mathring{R}\,(W^{S_{0}^2}),W^{S_{0}^{2}} \right\rangle +\frac{4(n-8)}{n}R_{lt}W_{ijkl}W_{ijkt}+4\sum\limits_{i,j}{<\mathring{R}(S^{ij}),S^{ij}>}+\frac{16}{n^2}s|W|^2.
		\end{align*}
		This completes the proof of Proposition~\ref{pro2}.
	\end{proof}
	
	Let $\{S^{\alpha}\}$ is an orthonormal eigenbasis for $\mathring{R}$ with corresponding eigenvalues $\{\lambda_{\alpha}\}$. Since $S^{ij}$ is a traceless symmetric $(0,2)$-tensor, we have
	$$\sum_{i,j}\left\langle\mathring{R}(S^{ij}),S^{ij}\right\rangle=\sum_{\alpha}\left[\lambda_{\alpha}\sum_{i,j}{(S_{\alpha}^{ij})^2}\right],$$
	where $S_{\alpha}^{ij}$ is the component of $S^{ij}$ under the basis $\{S^{\alpha}\}$. Hence Proposition~\ref{pro2} can be written as 
	\begin{equation}\label{Bochner}
		3\left\langle \Delta W,W \right\rangle=  \sum_{\alpha}{\lambda_{\alpha}\left|S^{\alpha}W\right|^2}+\frac{4(n-8)}{n}R_{lt}W_{ijkl}W_{ijkt}+4\sum_{\alpha}\left[\lambda_{\alpha}\sum_{i,j}{(S_{\alpha}^{ij})^2}\right]+\frac{16}{n^2}s|W|^2.
	\end{equation}
	
	Based on the calculations presented in \cite{DF24}, it is shown that
	\begin{eqnarray*}
		\sum\limits_{\alpha=1}^{N}{\left|S^{\alpha}W \right|^2}=\frac{2(n^2+n-8)}{n}{\left| W \right|^2},
	\end{eqnarray*}
	and $$\left| S^{\alpha}W \right|^2 \le\frac{8(n-2)}{n}\left| W \right|^2.$$ 
	Therefore,
	\begin{equation}\label{estimate}
		\sum\limits_{\alpha}{{\lambda }_{\alpha}\left| S^{\alpha}W \right|^2}\ge [\mathring{R},\frac{8(n-2)}{n},\frac{2(n^2+n-8)}{n}]|W|^2.
	\end{equation}
	
	According to \cite[Example 3.2]{NPW23}, the scalar curvature $s$ of $M$ satisfies
	\begin{align}
		s\ge \frac{2n}{n+2}[\mathring{R},1,\frac{(n-1)(n+2)}{2}],\label{scalar}
	\end{align}
	and from \cite[Lemma 3.14]{NPW23}, the Ricci curvature $Ric$ of $M$ satisfies
	\begin{align*}
		\Ric \geq \frac{n-1}{n+1} \left[ \mathring{R}, 1, n \right] + \frac{1}{n(n+1)}s.
	\end{align*}
	Combining (\ref{scalar}) with Lemma~\ref{ComputationsWithWeights}(c), we obtain
	\begin{align}
		\Ric &\geq \frac{n-1}{n+1} \left[ \mathring{R}, 1, n \right] + \frac{2}{(n+1)(n+2)}[\mathring{R},1,\frac{(n-1)(n+2)}{2}]\nonumber\\
		&\ge \left[ \mathring{R}, \frac{n}{n+2}, n-1 \right].\label{estimate2}
	\end{align}	
	Furthermore, applying (\ref{estimate1}) to the curvature tensor $R_{ijkl}=\delta_{ik}\delta_{jl}-\delta_{il}\delta_{jk}$ gives
	$$\sum_{\alpha}{\sum_{i,j}{(S_{\alpha}^{ij})^2}}=3|W|^2.$$
	From the weight principle, it follows that
	\begin{align}
		\sum_{\alpha}\left[\lambda_{\alpha}\sum_{i,j}{(S_{\alpha}^{ij})^2}\right] \ge 3\left[\mathring{R},1,1 \right]\left| W \right|^2.\label{estimate3}
	\end{align}	
	
	\begin{proposition}\label{2.3}
		Let $M$ be an $n(\ge 8)$-dimensional complete Riemannian manifold with harmonic Weyl curvature. Then
		$$3\left\langle \Delta W,W \right\rangle \ge [\mathring{R},\frac{8(3n-1)}{n+2},6(n-1)]\left| W \right|^2.$$
	\end{proposition}
	\begin{proof}
		Combining Proposition~\ref{pro2} with (\ref{scalar}),(\ref{estimate2}),(\ref{estimate3}) and Lemma~\ref{ComputationsWithWeights} yields 
		\begin{align*}
			3&\left\langle \Delta W,W \right\rangle =\left\langle \mathring{R}\,(W^{S_{0}^2}),W^{S_{0}^{2}} \right\rangle +\frac{4(n-8)}{n}R_{lt}W_{ijkl}W_{ijkt}+4\sum\limits_{i,j}{<\mathring{R}(S^{ij}),S^{ij}>}+\frac{16}{n^2}s|W|^2 \\ 
			\ge& [\mathring{R},\frac{8(n-2)}{n},\frac{2(n^2+n-8)}{n}]\left| W \right|^2+\frac{4(n-8)}{n}[\mathring{R},\frac{n}{n+2},n-1]\left| W \right|^2+12[R,1,1]\left| W \right|^2\\
			&+\frac{32}{n(n+2)}[R,1,\frac{(n-1)(n+2)}{2}]\left| W \right|^2 \\ 
			\ge&[\mathring{R},\frac{8(n-2)}{n}+\frac{4(n-8)}{n+2}+12+\frac{32}{n(n+2)},\frac{2(n^2+n-8)}{n}+\frac{4(n-8)(n-1)}{n}+12+\frac{16(n-1)}{n}]\left| W \right|^2 \\ 
			=&[\mathring{R},\frac{8(3n-1)}{n+2},6(n-1)]\left| W \right|^2 .
		\end{align*}
	\end{proof}
	
	\begin{proof}[\textbf{Proof of Theorem~\ref{A}}] 
		
		If $\mathring{R}$ is $\frac{3(n-1)(n+2)}{4(3n-1)}$-nonnegative, then by using Proposition~\ref{2.3} and  Lemma~\ref{Lem 3.4},	  we have
		\[\frac{1}{2}\Delta \left| W \right|^2=\left| \nabla W \right|^2+\left\langle \Delta W,W \right\rangle\ge 0. \]
		By the maximum principle, from the above we have $\nabla W=0.$ 
		According to Derdzi\'{n}ski and Roter’s result \cite{DR} that Riemannian manifolds with $\nabla W=0$ are either locally conformally flat or locally symmetric, we see that $(M, g)$ is either locally conformally flat or locally symmetric.
		
		By Theorem 1.8 in \cite{Li24}, $(M,g)$ is either locally irreducible or flat.
		
		If $(M,g)$ is locally conformally flat, then by the classification theorem in \cite{CH,Z} and the fact that $(M,g)$ is locally irreducible, it follows that $(M,g)$ is globally conformally equivalent to a space of positive constant curvature or is isometric to a flat manifold.
		
		If $(M,g)$ is locally symmetric, then it has parallel Ricci curvature. Since $(M,g)$ is locally irreducible, it is an Einstein manifold. Results from  \cite{DF24,FLD} imply that $M$ is necessarily of constant curvature. Therefore, $(M,g)$ is a space of constant curvature.
	\end{proof}
	
	\section{The case for $n=5$ and $n=7$}
	In this section, we prove Theorem~\ref{B}. First, we establish an upper bound for the Ricci curvature in terms of the scalar curvature and the curvature operator of the second kind.
	\begin{proposition}\label{3.1}
		Let $R$ be an algebraic curvature tensor on a Euclidean vector space $V$
		of dimension $n \ge 3$. If $R$ has $\frac{(n-1)(n-2)}{2}$-nonnegative curvature operator of the second kind, then the Ricci curvature of $R$ satisfies $Ric\le \frac{s}{2}.$
	\end{proposition}
	\begin{remark}This proposition follows from the proof of \cite[Theorem 1.6]{Li22}.
	\end{remark}
	\begin{proof}
		Let $\{e_1, . . . ,e_n\}$ be an orthonormal basis of $V$ and define:
		$${\psi}_{kl}=\frac{1}{\sqrt{2}}e_{k}\odot e_l,\quad\text{for}\quad 2\le k<l\le n.$$
		Then $\{{\psi}_{kl}\}$ is an orthonormal subset of $S_0^2(V)$. Since $R$ has $\frac{(n-1)(n-2)}{2}$-nonnegative curvature operator of the second kind, we can calculate that 
		$$0\le \sum\limits_{2\le k<l\le n}{\left\langle \mathring{R}({\psi}_{kl}),{\psi}_{kl} \right\rangle} =\sum\limits_{2\le k<l\le n}{R_{klkl}}=\frac{s}{2}-R_{11},$$
		which implies that $Ric \le \frac{s}{2}$.
	\end{proof}
	\begin{proof}[\textbf{Proof of Theorem~\ref{B}}]
		For $n=5$,   we have 
		\begin{align*}
			\alpha (W)=&{{R}_{sjti}}{{W}_{sjkl}}{{W}_{tikl}}\\
			=&{{W}_{sjti}}{{W}_{sjkl}}{{W}_{tikl}}+\frac{1}{n-2}({{R}_{st}}{{\delta }_{ji}}+{{R}_{ji}}{{\delta }_{st}}-{{R}_{si}}{{\delta }_{jt}}-{{R}_{jt}}{{\delta }_{si}}){{W}_{sjkl}}{{W}_{tikl}} \\ 
			&-\frac{s}{(n-1)(n-2)}({{\delta }_{st}}{{\delta }_{ji}}-{{\delta }_{si}}{{\delta }_{jt}}){{W}_{sjkl}}{{W}_{tikl}} \\ 
			=&{{W}_{sjti}}{{W}_{sjkl}}{{W}_{tikl}}+\frac{4}{n-2}{{R}_{st}}{{W}_{sjkl}}{{W}_{tjkl}}-\frac{2s}{(n-1)(n-2)}{{\left| W \right|}^{2}}
		\end{align*}
		and
		\begin{align*}
			\beta (W)=&{{R}_{sikt}}{{W}_{sjkl}}{{W}_{ijtl}}\\
			=&{{W}_{sikt}}{{W}_{sjkl}}{{W}_{ijtl}}+\frac{1}{n-2}({{R}_{sk}}{{\delta }_{it}}+{{R}_{it}}{{\delta }_{sk}}-{{R}_{st}}{{\delta }_{ik}}-{{R}_{ik}}{{\delta }_{st}}){{W}_{sjkl}}{{W}_{ijtl}} \\ 
			&-\frac{s}{(n-1)(n-2)}({{\delta }_{st}}{{\delta }_{ji}}-{{\delta }_{si}}{{\delta }_{jt}}){{W}_{sjkl}}{{W}_{ijtl}} \\ 
			=&{{W}_{sikt}}{{W}_{sjkl}}{{W}_{ijtl}}-\frac{2}{n-2}{{R}_{st}}{{W}_{sjkl}}{{W}_{kjtl}}+\frac{s}{(n-1)(n-2)}{{W}_{sjkl}}{{W}_{kjsl}} \\ 
            =&{{W}_{sikt}}{{W}_{sjkl}}{{W}_{ijtl}}-\frac{2}{n-2}{{R}_{st}}{{W}_{sjkl}}(-{{W}_{jtkl}}-{{W}_{tkjl}})+\frac{s}{(n-1)(n-2)}{{W}_{sjkl}}(-{{W}_{jskl}}-{{W}_{skjl}}) \\ 
			=&{{W}_{sikt}}{{W}_{sjkl}}{{W}_{ijtl}}-\frac{1}{n-2}{{R}_{st}}{{W}_{sjkl}}{{W}_{tjkl}}+\frac{s}{2(n-1)(n-2)}{{\left| W \right|}^{2}}. \\ 
		\end{align*}
		Due to Jack and Parker’s result \cite{JP}, when $n\le 5$
		\[{{W}_{sjti}}{{W}_{sjkl}}{{W}_{tikl}}=2{{W}_{sikt}}{{W}_{sjkl}}{{W}_{ijtl}},\]
		hence
		$$\alpha (W)=2\beta (W)+\frac{6}{n-2}R_{st}W_{sjkl}W_{tjkl}-\frac{3}{(n-1)(n-2)}s\left| W \right|^2.$$
		Combining above with (\ref{laplacian}), we have 
		\begin{equation}\label{5Bochner}
			\left\langle \Delta W,W \right\rangle =-6\beta (W)-\frac{2(n-5)}{n-2}{{R}_{st}}{{W}_{sjkl}}{{W}_{tjkl}}+\frac{3}{(n-1)(n-2)}s\left| W \right|^2.
		\end{equation}
		Consequently, it follows from (\ref{estimate1}) and Proposition~\ref{pro1} that 
		\begin{equation}\label{5Bochner2}
			\sum\limits_{i,j}{\left\langle R(S^{ij}),S^{ij} \right\rangle}=-3\beta (W)+\frac{3}{n-2}R_{st}W_{sjkl}W_{tjkl}-\frac{3}{2(n-1)(n-2)}s\left| W \right|^2,
		\end{equation}
		and
		\begin{equation}\label{5Bochner3}
			\left\langle R(W^{S_{0}^{2}}),W^{S_{0}^{2}} \right\rangle =-6\beta (W)+\frac{2(n^2-n-32)}{n(n-2)}R_{st}W_{sjkl}W_{tjkl}-\frac{n^2-48n+32}{n^2(n-1)(n-2)}s\left| W \right|^2.
		\end{equation}
		Substituting (\ref{5Bochner2}) and (\ref{5Bochner3}) into (\ref{5Bochner}), we get for $n=5$
		\[\left\langle \Delta W,W \right\rangle =\frac{5}{9}\left\langle R(W^{S_{0}^{2}}),W^{S_{0}^{2}} \right\rangle +\frac{8}{9}\sum\limits_{i,j}{\left\langle R(S^{ij}),S^{ij} \right\rangle}+\frac{1}{45}s\left| W \right|^2.
		\]
		
		For $n=7$, If $\mathring{R}$ is $\frac{54}{35}$-nonnegative, form Proposition~\ref{3.1} we have $\Ric\le \frac{s}{2}.$ This together with (\ref{Bochner}) yields
		\[3\left\langle \Delta W,W \right\rangle\ge  \sum_{\alpha}{\lambda_{\alpha}\left|S^{\alpha}W\right|^2}+4\sum_{\alpha}\left[\lambda_{\alpha}\sum_{i,j}{(S_{\alpha}^{ij})^2}\right]+\frac{2}{49}s|W|^2.\]
		The remainder of the argument is analogous to the proof of Theorem~\ref{A}, which completes the proof of Theorem~\ref{B}.
	\end{proof}
	
	\section{Rigidity theorems for dimensional four}
	\begin{proof}[\textbf{Proof of Theorem~\ref{C}}]
		The Weitzenböck formula (cf. 16.73 in \cite{Einstein}) on $4$-dimensional manifolds with harmonic Weyl curvature is given by
		\[\frac12\Delta \left| W^+ \right|^2=\left| \nabla W^+ \right|^2+\frac{s}{2}\left| W^+ \right|^2-72\det W^+.\]
		
		Denote the eigenvalues of $W^+$ by $a_{1} \le a_{2} \le a_{3}$, and those of $W^-$ by $b_{1} \le b_{2} \le b_{3}$.
		
		According to \cite{CGT23}, on 4-manifold, there is an
		orthonormal basis of $S_0^2(TM^4)$ such that all diagonal elements of the matrix of $\mathring{R}$ are given by
		\[{\lambda}_{ij}=\frac{s}{12}-a_i-b_j,\quad i,j=1,2,3.\]
		If  $\mathring{R}$ satisfies (\ref{cone}), then
		\[a_3=\frac{s}{12}-\frac{\lambda_{31}+\lambda_{32}+\lambda_{33}}{3}\le \frac{s}{12}-\frac{\lambda_{1}+\lambda_{2}+\lambda_{3}}{3}\le\frac{s}{12}+\bar{\lambda}=\frac{s}{6},\]
		and \[b_3=\frac{s}{12}-\frac{\lambda_{13}+\lambda_{23}+\lambda_{33}}{3}\le \frac{s}{12}-\frac{\lambda_{1}+\lambda_{2}+\lambda_{3}}{3}\le \frac{s}{12}+\bar{\lambda}=\frac{s}{6}.\]
		Consequently, $s\ge0$;  moreover,  if $s=0$, then $W^\pm=0$, i.e., $M$ is conformal flat. 
		
		Let
		$$f(a_1,a_2,a_3)=\frac{s}{2}\left| W^+ \right|^2-72\det W^+=2\left[s(a_{1}^2+a_2^2+a_{3}^2)-36a_1a_2a_3\right].$$ 
		Since $a_1+a_2+a_3=0$, we obtain
		\begin{align*}
			f=2\left[s(a_{1}^2+a_{2}^2+a_{3}^2)-12(a_{1}^3+a_{2}^3+a_{3}^3)\right].  
		\end{align*}
		
		Now we consider the case $s>0$. The expression for $f$ can be rewritten as
		\begin{align*}
			& f=2\left[s(a_{1}^2+a_{2}^2+a_{3}^2)-12(a_{1}^3+a_{2}^3+a_{3}^3)\right] \\ 
			& =2\left[s\sum\limits_{i=1}^{3}{(\frac{s}{6}-a_i-\frac{s}{6})^2}+12\sum\limits_{i=1}^{3}{(\frac{s}{6}-a_i-\frac{s}{6})^3}\right]  \\ 
			& =2\left[12\sum\limits_{i=1}^{3}{(\frac{s}{6}-a_i)^3}-5s\sum\limits_{i=1}^{3}{(\frac{s}{6}-a_i)^2}+\frac{s^3}{4}\right]  \\ 
			& =\frac{125}{72}s^3\left[ \sum\limits_{i=1}^{3}{[\frac{12}{5s}(\frac{s}{6}-a_i)]^3}-\sum\limits_{i=1}^{3}{[\frac{12}{5s}(\frac{s}{6}-a_i)]^2} \right]+\frac{s^3}{2}.
		\end{align*}
		Since 
		\[\frac{12}{5s}(\frac{s}{6}-a_i)\ge 0\quad\text{and}\quad \sum\limits_{i}{\frac{12}{5s}(\frac{s}{6}-a_i)}=\frac{6}{5},\]
		it follows from \cite[Lemma 3.3]{FL} that the minimal points of $f$ are 
		$$(\frac{12}{5s}(\frac{s}{6}-a_1),\frac{12}{5s}(\frac{s}{6}-a_2),\frac{12}{5s}(\frac{s}{6}-a_3))=(\frac{2}{5},\frac{2}{5},\frac{2}{5})$$
		and  
		$$(\frac{12}{5s}(\frac{s}{6}-a_1),\frac{12}{5s}(\frac{s}{6}-a_2),\frac{12}{5s}(\frac{s}{6}-a_3))=(\frac{3}{5},\frac{3}{5},0),$$
		which correspond to $$(a_1,a_2,a_3)=(0,0,0) \quad\text{and}\quad (a_1,a_2,a_3)=(-\frac{s}{12},-\frac{s}{12},\frac{s}{6}),$$
		respectively.
		Consequently, we have
		$$f\ge f(0,0,0)=f(-\frac{s}{12},-\frac{s}{12},\frac{s}{6})=0,$$
		which in this case implies
		\[\frac12\Delta \left| W^+ \right|^2=\left| \nabla W^+ \right|^2+\frac{s}{2}\left| W^+ \right|^2-72\det W^+\geq0.\]
		By the maximum principle, we obtain $ \nabla W^+=0$, and either $(a_1,a_2,a_3)=(0,0,0),$ or $(a_1,a_2,a_3)=(-\frac{s}{12},-\frac{s}{12},\frac{s}{6}).$ 
		
		Similarly, we obtain $\nabla W^-=0$, with $(b_1,b_2,b_3)$ being either $(0,0,0)$ or $(-\frac{s}{12},-\frac{s}{12},\frac{s}{6}).$ 
		It follows that either $W=0$, or $W\neq0$ and $ \nabla W=0$--the latter case implying that $M$ is a Kähler manifold of constant scalar curvature \cite{D}. Together with the harmonic Weyl curvature condition, this forces the Ricci tensor to be parallel. Consequently, we have $ \nabla R=0$; that is, $M$ is locally symmetric. So when the universal cover $\tilde{M}$ of $M$ is irreducible, $\tilde{M}$ is  $\mathbb{CP}^2$ with the Fubini–Study metric by the classification of $4$-dimensional symmetric spaces; when the universal cover $\tilde{M}$ of $M$ is reducible, then $\Sigma_1^2\times \Sigma_2^2$ with the product metric, where $k_i(i=1,2)$ is constant Gaussian curvature of $\Sigma_i^2$  and $k_1+k_2=\frac{s}{2}>0$.
	\end{proof}
	\begin{corollary}\label{D}
		Let $(M,g)$ be a $4$-dimensional compact manifold with harmonic Weyl curvature and $\chi(M)\geq0$ whose the curvature operator of the second kind satisfies 
		\begin{equation*}
			\lambda_{1}+\lambda_{2}+\lambda_{3}\ge -3\bar{\lambda},
		\end{equation*}
		then  one of the following must be true:\\
		(1) $M$ is flat.\\
		(2) $M$ is conformal to a quotient of the round sphere $(S^4, g_0)$.\\
		(3) $M$ is conformal to a quotient of  $\mathbb{S}^1\times \mathbb{S}^{3}$ with the product metric.\\
		(4) up to rescaling, $M$ is $\mathbb{CP}^2$ with the Fubini–Study metric.\\
		(5) up to rescaling,  $M$ is isometric to a quotient of $\Sigma_1^2\times \Sigma_2^2$ with the product metric, where $k_i(i=1,2)$ is constant Gaussian curvature of $\Sigma_i^2$  and $\frac{s}{2}=k_1+k_2>0$.
	\end{corollary}
	\begin{proof}If $W=0$ and $s=0$, the Chern–Gauss–Bonnet theorem implies that $\mathring{Ric}=0$, i.e., $M$ is flat.
		
		In the case of $W=0$ and $s>0$, then \cite[Proposition F]{F} implies that the second Betti number  $b_2(M)=0$, and $0\leq\chi(M)\leq2$, i.e., the first Betti number  $b_1(M)=0$ or $1$. It is derived from the Chern–Gauss–Bonnet theorem that
		$$-2\int_M|\mathring{Ric}|^2+\frac{1}{6}\int_M s^2=32{\pi}^2\chi(M)\geq0, \text{i.e.}, \int_M|\mathring{Ric}|^2\leq\frac{1}{6}\int_M s^2.$$
		By \cite[Theorem 1.7]{FX},  $M$ is conformal to a quotient of the stand sphere $\mathbb{S}^4$  or $\mathbb{S}^1\times \mathbb{S}^{3}$ with the product metric. Combining Theorem~\ref{C}, we complete the proof  of  Corollary~\ref{D}.
	\end{proof}
	\begin{remark}
		For  $4$-dimensional manifolds with harmonic curvature, in  (2) and (3) of Corollary~\ref{D},  the term ``conformal" can be strengthened to ``isometric". In fact, $M$ is isometric to a quotient of the stand sphere $\mathbb{S}^4$ since the stand sphere $\mathbb{S}^4$ is Einstein. On the other hand,  $\chi(M)=0$ implies that $\int_M|\mathring{Ric}|^2=\frac{1}{12}Y^2(M,[\tilde{g}])$ since $-2\int_M|\mathring{Ric}|^2+\frac{1}{6}\int_M s^2$ is conformal invariant. By \cite[Theorem 1.6]{FX},  $M$ is isometric to a quotient of  $\mathbb{S}^1\times \mathbb{S}^{3}$ with the product metric.
	\end{remark}
	
	\begin{thm}\label{E}
		Let $(M,g)$ be a $4$-dimensional compact manifold with harmonic Weyl curvature and $4\frac12$-nonnegative curvature operator of the second kind.
		Then  one of the following must be true:\\
		(1) $M$ is flat.\\
		(2)  $M$ is conformal to a quotient of the round sphere $(\mathbb{S}^4, g_0)$.\\
		(3) $M$ is conformal to a quotient of  $\mathbb{R}^1\times \mathbb{S}^{3}$ with the product metric.\\
		(4)  up to rescaling, $M$ is isometric to $\mathbb{CP}^2$ with the Fubini–Study metric.
	\end{thm}
	\begin{remark}
		For general $4$-dimensional manifolds, \cite[Theorem 1.4]{Li22} yields a similar but weaker result than the one stated in this theorem. For  $4$-dimensional manifolds with harmonic curvature, in  (2) and (3) of Theorem~\ref{E},  the term ``conformal" can be strengthened to ``isometric".  This can be obtained similarly to the latter part of the proof of Corollary~\ref{D} that $M$ is isometric to a quotient of either the round sphere $\mathbb{S}^4$ or $\mathbb{R}^1\times \mathbb{S}^{3}$.
	\end{remark}
	\begin{proof}
		Suppose that ${\lambda }_1\le {\lambda }_2\le \cdots \le {\lambda }_{9}$
		are the eigenvalues of $\mathring{R}\,$ and $\bar{\lambda }$ is their average. Then $4\frac12$-nonnegative curvature operator of the second kind $\mathring{R}\,$ satisfies
		\begin{align*}
			& {\lambda }_1+{\lambda }_2+{\lambda }_3\geq-{\lambda }_4-\frac12 {\lambda }_5\\ 
			&\geq-\frac{\frac12 {\lambda }_5+{\lambda }_6+{\lambda }_7+{\lambda }_8+{\lambda }_9}{3} \\ 
			& \geq-\frac{{\lambda }_1+{\lambda }_2+{\lambda }_3+{\lambda }_4+{\lambda }_5+{\lambda }_6+{\lambda }_7+{\lambda }_8+{\lambda }_9}{3} \\ 
			& =-3\bar{\lambda }.
		\end{align*}
		Combining Theorem~\ref{C}, $M$ is either flat, conformal flat, or is isometric to $\mathbb{CP}^2$ with the Fubini–Study metric  because $\mathbb{S}^2\times \mathbb{S}^2$ has $4\frac12$-negative curvature operator of the second kind.
		By \cite[Theorem 1.6]{Li22}, $M$ has nonnegative Ricci curvature. Then, applying  \cite[Theorem A]{CH} and \cite[Theorem 1]{Z}, we conclude that $M$ is conformal to a quotient of either the round sphere $\mathbb{S}^4$ or $\mathbb{R}^1\times \mathbb{S}^{3}$.
		This completes the proof of Theorem~\ref{E}.
	\end{proof}

\end{document}